\documentclass[11pt]{article}

\usepackage{graphicx}
\usepackage{url}
\usepackage{amsmath}
\usepackage{amsfonts}
\usepackage{verbatim}
\usepackage{amssymb}
\usepackage{amsthm}
\usepackage{color}
\usepackage{float}

\newtheorem{thm}{Theorem}
\newtheorem{defn}[thm]{Definition}
\newtheorem{prop}[thm]{Proposition}
\newtheorem{cor}[thm]{Corollary}
\newtheorem{lem}[thm]{Lemma}
\newtheorem{remark}[thm]{Remark}

\def\mcP{{\mathcal{P}}}

\def\g{{\gamma}}
\def\a{{\alpha}}
\def\be{{\beta}}

\newcommand{\bit}{\begin{itemize}}
\newcommand{\eit}{\end{itemize}}
\newcommand{\ben}{\begin{enumerate}}
\newcommand{\een}{\end{enumerate}}
\newcommand{\beq}{\begin{equation}}
\newcommand{\eeq}{\end{equation}}
\newcommand{\bea}{\begin{eqnarray*}}
\newcommand{\eea}{\end{eqnarray*}}
\newcommand{\bpf}{\begin{proof}}
\newcommand{\epf}{\end{proof}\ms}
\newcommand{\ms}{\medskip}

\begin{document}
\title{Limit and Morse Sets for Deterministic Hybrid Systems}

\author{
Kimberly Ayers \thanks{Department of Mathematics, Iowa State University, Ames, IA 50014, USA ($\mbox{kdayers@iastate.edu}$). Research supported by DMS 0502354.}
 \and Xavier Garcia\thanks{Department of Mathematics, University of Minnesota - Twin Cities, Minneapolis, MN 55455, USA
($\mbox{garci363@umn.edu}$). Research supported by DMS 0750986 and DMS 0502354.} \and Jennifer Kunze\thanks{
Department of Mathematics, Saint Mary's College of Maryland, St Marys City, MD 20686, USA
($\mbox{jckunze@smcm.edu}$). Research supported by DMS 0750986. } \and Thomas Rudelius\thanks{
Department of Mathematics, Cornell University, Ithaca, NY 14850, USA
(twr27@cornell.edu). Research supported by DMS 0750986. } \and  Anthony Sanchez\thanks{%
Department of Mathematics, Arizona State University, Tempe, AZ 85281, USA
($\mbox{Anthony.Sanchez.1@asu.edu}$). Research supported by DMS 0750986 and DMS 0502354. } \and  Sijing Shao\thanks{
Department of Mathematics, Iowa State University, Ames, IA 50014, USA
(sshao@iastate.edu). Research supported by Iowa State University. } \and Emily Speranza\thanks{
Department of Mathematics, Carroll College, Helena, MT 59625, USA
($\mbox{esperanza@carroll.edu}$). Research supported by DMS 0750986. }
}

\date{}
\maketitle


\section{Introduction}

The theory of dynamical systems is used in many fields of study to model the evolution of systems over time.  However, most objects of real-world interest are too complicated to be modeled exactly by a dynamical system.  The random fluctuations that occur in nature are conventionally disregarded when analysis is performed, a practice that may be dubious in some situations.

In \cite{ourpaper}, we have considered hybrid systems that consist of a set of continuous time dynamical systems over a compact space, where the dynamical system acting at each time is determined by Markov chain.  We have taken $\{Z_n\}$ to be a randomly determined sequence of Markov states in some finite state space, and we have used this construction to find invariant probability measures.

In this paper, we change gears and instead treat $\{Z_n\}$ as a preselected, deterministic sequence of states.  This approach to Markov chains has been studied in some detail in \cite{discretesystems}.  We will adapt the results to our continuous time hybrid system, and we will see that under this new formulation, the hybrid system can be treated as a dynamical system rather than a stochastic one.  We begin by trying to understand the space on which this dynamical system lives, showing that it is compact.  As a result, we can discuss limit sets and Morse decompositions of the hybrid dynamical system, but we find that the limiting behavior of the system can be highly irregular.  We explore three simple examples, illustrating the variety that is possible in the limiting behavior of hybrid dynamical systems.


\section{Directed Graphs and Dynamical Systems on Graphs}

The majority of the definitions, theorems, and notation in this section were borrowed from \cite{discretesystems}. The proof of all of these theorems can be found there as well. We introduce basic definitions for directed graphs associated with symbolic dynamical systems that will be used in the remainder of the paper.

\subsection{Directed Graphs}

\begin{defn}\label{def 1} A finite directed graph $G=(V,E)$ is a pair of sets $V=\{1,...,n\}$ called vertices, and $E\subseteq V \times V$, called edges.
\end{defn}

\begin{defn}\label{def 3} Any graph $G=(V,E)$ has a set $\mcP = \{(x_1,...,x_k),|(x_i,x_{i+1})\in E \ ; \ i,k \in \mathbb{N}\}$. Any element of $\mcP$ is called an admissible path of G.
\end{defn}

\begin{defn}\label{def 6}The out-degree of any $\a\in V$, denoted  $o(\a)$, is the number of $\g\in\mcP$ with length 1 and $\g_1=\a$.
\end{defn}

\begin{defn}\label{def 7}The in-degree of any $\a\in V$, denoted  $i(\a)$, is the number of $\g\in\mcP$ with length 1 and $\g_F=\a$.
\end{defn}

\begin{defn}\label{defNgraph} A finite directed graph $G=(V,E)$ with $o(\a)\geq1, i(\a)\geq1$; for all $\a\in V$ is called an $N$-graph.
\end{defn}

\begin{defn}\label{def 13} A communicating class in an $N$-graph $G=(V,E)$ is a subset $C\subseteq V$ for which two things are true:

        \begin{enumerate}
            \item For all $\a,\be\in C$ there exists $\g\in\mcP$ such that $\g_1=\a$ and $\g_F=\be$.

            \item There exists no $C'\supset C$ where for all $\a',\be'\in C'$ there exists $\g'\in\mcP$ such that $\g'_1=\a'$ and $\g'_F=\be'$. This condition is called maximality.
        \end{enumerate}
\end{defn}

\begin{defn} Communicating classes can be classified further in two ways:

        \begin{enumerate}
            \item A communicating class $C$ is variant if there exists $\g\in\mcP$ with $\g_1\in C$ and $\g_F\not\in C$.
            \item A communicating class $C$ is invariant if for all $\g\in\mcP$ with $\g_1\in C$, $\g_F\in C$.
        \end{enumerate}
\end{defn}

\begin{remark}
Note that the empty set is not a communicating class because the empty set does not meet the requirements for maximality.
\end{remark}

\begin{thm}\label{thm 3}Every $N$-graph contains an invariant communicating class.
\end{thm}

\subsection{Dynamical Systems}
\begin{defn}
 A dynamical system (d.s.) on a metric space $X$ is given by a map $\Phi : \, \mathbb{T} \times X \rightarrow X $ that satisfies $\Phi(0,x)=x$ and $\Phi(t+s,x)=\Phi(t,\Phi(s,x))$ for all $x\in X$ and all $t,s\in \mathbb{T}$.  $\Phi$ can be expressed by two different but equivalent notations for $x, \,  x' \in X$ and $t\in \mathbb{T}$: $$\Phi(t,x)=  x' \, or\, \Phi_t(x)=  x' $$ {\rm \cite{Kliemann}}.
\end{defn}

\begin{defn}
A d.s. is 1-sided when $\mathbb{T}=\mathbb{N}$ or $\mathbb{T}=\mathbb{R} ^+$.
A d.s. is 2-sided when $\mathbb{T}=\mathbb{Z}$ or $\mathbb{T}=\mathbb{R}$.
\end{defn}

\begin{lem}
Any 2-sided d.s. with mapping $\Phi_t$ has an inverse mapping $\Phi_{-t}$ where $$\Phi_t \circ \Phi_{-t}(x)=\Phi_{-t} \circ \Phi_{t}(x)=Id(x)=x$$
\end{lem}

\begin{defn}\label{omegalimit1} 
The $\omega$-limit set of an element $x\in X$ is
$$\omega (x) = \{ y\in X | (\exists \, t_k \rightarrow\infty, k \in \mathbb{N})(\Phi (t_k,x)\rightarrow y)\}.$$
The $\alpha$-limit set of an element $x\in X$ is
$$\alpha (x) = \{ y \in X | (\exists \, t_k \rightarrow -\infty, k \in \mathbb{N})(\Phi (t_k,x)\rightarrow y)\}.$$
\end{defn}

\begin{defn}
The $\omega$-limit set of a subset $Y \subseteq X$ is given by
$$
\omega(Y) = \{ y \in X | (\exists \, t_k \rightarrow\infty, y_k \in Y, k \in \mathbb{N})(\Phi (t_k,y_k)\rightarrow y)\}.
$$
The $\alpha$-limit set of a subset $Y \subseteq X$ is given by
$$
\alpha(Y) = \{ y \in X | (\exists \, t_k \rightarrow -\infty, y_k \in Y, k \in \mathbb{N})(\Phi (t_k,y_k)\rightarrow y)\}.
$$\end{defn}

\subsection{Shift Spaces}
 
\begin{defn}  Given an $N$-graph $G=(V,E)$, the bi-infinite product space $\Upsilon$ of the set $V=\{1,...,n\}$ is the set of all bi-infinite sequences $x=(...x_{-1},x_{0},x_1,...)$ where $x_i\in V$ for all $i\in \mathbb{Z}$. \end{defn}

\begin{defn} Given an $N$-graph $G=(V,E)$ with $A\subset V$ and $\a\in V$, we define:
    \begin{itemize}
        \item $\Omega=\{ \ (...,x_{-2},x_{-1},x_0,x_1,x_2,...) \ | \ (x_i,x_{i+1})\in E \ \}$ to be the shift space of $G$.
        \item $\Omega_A=\{ \ x\in\Omega \ | \ x_i\in A $ for all $ i\in\mathbb{Z} \ \}$ to be the lift of $A$.
    \end{itemize}
\end{defn}

\noindent The flow on this dynamical system is determined by the left shift mapping $\Phi$, defined in {\rm \cite{discretesystems}.  It is in this paper that a metric is defined, and it is shown that this shift operator is continuous, and that $\Omega$ is compact.








\section{Generalization of the Shift Space to Continuous Time}

\subsection{$\bar \Delta$ and $\Delta$}
Note that in {\rm \cite{discretesystems}}, the flow on $\Omega$ is a discrete time dynamical system.  However, to insert this behavior into another dynamical systems to create a hybrid system, it requires that we extend this system to a continuous time dynamical system.  The obvious extension of a sequence into a function on $\mathbb{R}$ is a piecewise constant function.

\begin{defn}\label{lambda}
Let 
\begin{center}
$\bar{\Lambda} =  \left\{ x : \mathbb{R} \rightarrow S \left|
\begin{array}{cc}
\{x(ih)\}_{i \in \mathbb{Z}} \in \Upsilon \\
x$ is piecewise constant on $[nh,(n+1)h) \, \forall \, n \in \mathbb{Z}
\end{array}  \right\} \right .$
\end{center}
and
$$\Lambda = \{ x(\cdot + t) | x \in \Lambda, t \in \mathbb{R}\} .$$
\end{defn}

\begin{defn}\label{delta}
Let 
\begin{center}
$\bar{\Delta} =  \left\{ x : \mathbb{R} \rightarrow S \left|
\begin{array}{cc}
\{x(ih)\}_{i \in \mathbb{Z}} \in \Omega \\
x$ is piecewise constant on $[nh,(n+1)h) \, \forall \, n \in \mathbb{Z}
\end{array}  \right\} \right.$
\end{center}
and 
$$\Delta = \{ x(\cdot + t) | x \in \Delta, t \in \mathbb{R}\} .$$
\end{defn}

\noindent
In other words, $\Lambda$ and $\Delta$ are the sets of functions that result from translating the functions in $\bar\Lambda$ and $\bar\Delta$, respectively, by some $t \in \mathbb{R}$.  We allow for all horizontal translations of functions in $\bar\Lambda$ and $\bar\Delta$ in order for the spaces to be closed under shifts by $t$ for all $t\in \mathbb{R}$.

The next definition adapts the shift operator to continuous time by taking functions in $\bar \Delta$ as a generalization of bi-infinite sequences in $\Omega$.

\begin{defn}\label{psi}
Let 

\begin{center}
$
\begin{array}{lr}
 \psi: \mathbb{R} \times \bar \Delta  \rightarrow \bar \Delta \\
 \,\,\,\,\,\,\,\,(t, x(\cdot)) \mapsto x(\cdot + t)
\end{array}$
\end{center}
\end{defn}

Note that $\psi$ satisfies the flow property:
$$\psi(s+t, x(k)) = x(k + s + t) = x((k + t) + s) = \psi(s, x(k + t)) = \psi(s, \psi(t, x(k))) .$$

We now impose a metric on the set of functions $\bar{\Delta}$.

\begin{defn}\label{bardeltafunction}
Define the function
\begin{center}
$\begin{array}{lc}
f: \bar{\Delta} \times \bar{\Delta} \times \mathbb{Z} \rightarrow \mathbb{R} \\ 
\,\,\,\,(x, y, i) \mapsto \frac{1}{h}\displaystyle\int_{ih}^{(i+1)h}{\delta(x,y,t)dt}
\end{array}$
\end{center}
where 
$$ \delta(x,y,t) =  \left \{ \begin{array}{cc}
1 & \mbox{if } x(t) \neq y(t)\\
0 & \mbox{if } x(t) = y(t)
\end{array} \right.$$
\end{defn}

\begin{thm}\label{bardeltafunctionmetric}
The function
\begin{center}
$
\begin{array}{lc}
 d: \bar{\Delta} \times \bar{\Delta} \rightarrow \mathbb{R} \\
 \,\,\,\,\,\,\,\,\,(x,y) \mapsto \displaystyle \sum_{i = - \infty}^{\infty}\frac{f(x,y,i)}{4^{|i|}}
\end{array} $
\end{center}
is a metric on $\bar{\Delta}$.
\end{thm}

\begin{proof}
\ \\
\begin{enumerate}

\item (Non-negativity) $f(x, x, i) = 0$, for all $i \in \mathbb{Z}$.  Therefore, $d(x,x) = 0$.  For $x \neq y$, $f(x,y,i) \neq 0$ for at least one $i \in \mathbb{Z}$, and $f(x, y, i) \geq 0$ for all $i \in \mathbb{Z}$.  Therefore, $d(x,y) > 0$ for all $x \neq y$.
\item (Symmetry) Clearly, $\delta(x,y,t) = \delta(y,x,t)$ for all $t \in \mathbb{R}, x, y \in \bar{\Delta}$.  So, $f(x,y,i) = f(y,x,i)$, $d(x,y) = d(y,x)$.
\item (Triangle inequality) Choose $x,y,z\in\bar\Delta$.  If $x=z$, then as $d(x,z)=0$ and $d$ is nonnegative, then clearly for all $y\in\bar\Delta$, $d(x,z)\leq d(x,y)+d(y,z)$.   

If $x\not =z$, then there exists $t\in\mathbb{R}$ such that $x(t)\not = z(t)$.  If $x(t) \neq z(t)$, then either $x(t) \neq y(t)$ or $y(t) \neq z(t)$.  Therefore, $\delta(x,z,t) = 1$ implies that $\delta(x,y,t) = 1$ and/or $\delta(y,z,t) = 1$, so $f(x,z,i) \leq f(x,y,i) + f(y,z,i)$.  Since $d$ is a linear combination of $f$'s, $d(x,z) \leq d(x,y) + d(y,z)$.
\end{enumerate}
\end{proof}

\begin{prop} 
The mapping 
$\sigma: \Omega \rightarrow \bar\Delta$ where $x\mapsto x(t)$ where $x(i)=x_i$ for all $i\in\mathbb{Z}$ is an isometric isomorphism.
\end{prop}

\begin{proof}
By the construction of $\bar\Delta$, $\sigma$ is clearly bijective.\\
To show that $\sigma$ is an isometry, it suffices to show that $f(x,y,i) = \bar{f}(x_i,y_i)$, where
$$ \bar{f}(x_i,y_i) =   
\left\{ \begin{array}{cc}
1 & \mbox{if } x(i) \neq y(i)\\
0 & \mbox{if } x(i) = y(i)
\end{array} \right.$$
since the bi-infinite sums for $d$ and $\bar{d}$ are identical.  Note that
$$f(x,y,i) = \frac{1}{h}\int_{ih}^{(i+1)h}{dt} = 1 = \bar{f}(x_i, y_i)$$
for $x(i) \neq y(i)$.
$$ f(x,y,i) = \frac{1}{h}\int_{ih}^{(i+1)h}{0 \cdot dt} = 0 = \bar{f}(x_i, y_i)$$
for $x(i) = y(i)$.
So indeed, $f(x,y,i) = \bar{f}(x_i,y_i)$, and $d(x,y) = \bar{d}(\{x_i\},\{y_i\})$.

\end{proof}


\begin{lem}\label{psicontinuous}
$\psi_t$ is continuous for all $t \in \mathbb{R}$.
\end{lem}

\begin{proof}
Given $x, y \in \bar{\Delta}$, we need to show that for all $\epsilon > 0$, there exists $\delta > 0$ such that
$$d(x,y) < \delta \Rightarrow d(\psi_t(x),\psi_t(y))<\epsilon.$$
Given any $\epsilon >0$, take $\delta = \frac{\epsilon}{4^n}$, where $n = \left\lceil|\frac{t}{h}|\right\rceil$, the least integer greater than the absolute value of $\frac{t}{h}$.  It is useful to rewrite $d(x,y)$ in the form
$$ d(x,y) = \frac{1}{h} \int_{-\infty}^{\infty}{\frac{1}{4^{\left |\lfloor \frac{t'}{h}\rfloor \right |}}\delta(x,y,t')dt'}$$
where $\delta$ is as defined above.  Given this, we can write
$$ d(\psi_t(x),\psi_t(y)) = \frac{1}{h} \int_{-\infty}^{\infty}{\frac{1}{4^{\left | \lfloor \frac{(t+t')}{h}  \rfloor \right |}}\delta(x,y,t')dt'}$$
And,
$$ \frac{1}{4^{\left | \lfloor \frac{(t+t')}{h} \rfloor \right |}} \leq 4^{\left \lceil |\frac{t}{h}| \right \rceil} \frac{1}{4^{\left | \lfloor \frac{t'}{h} \rfloor \right|}}$$

So, $$ d(\psi_t(x),\psi_t(y)) = \frac{1}{h} \int_{-\infty}^{\infty}{\frac{1}{4^{\left |\lfloor \frac{(t+t')}{h}  \rfloor \right |}}\delta(x,y,t')dt'}$$
$$ \leq 4^{\left \lceil |\frac{t}{h}| \right \rceil} \frac{1}{h} \int_{-\infty}^{\infty}{\frac{1}{4^{\left |\lfloor \frac{t'}{h} \rfloor \right|}}\delta(x,y,t')dt'}$$
$$ = 4^{\left \lceil |\frac{t}{h}| \right \rceil} d(x,y)$$
$$ < 4^{\left \lceil |\frac{t}{h}| \right \rceil} \delta$$
$$ = \epsilon$$
\end{proof}


\begin{lem}\label{deltabarcompact}
 $\Delta$ is compact.
\end{lem}

\begin{proof}
We will show that given any sequence $\{x^n\}, n \in \mathbb{N}$ of functions $x^n \in \Delta$, there exists a subseqence converging to some $x \in \Delta$.  To do this, we consider the space $\Delta$ to be the product of a circle of length $h$ with the set of allowable bi-infinite sequences, $S^1 \times \Omega$, where $S^1 \equiv \mathbb{R}$ mod $h$.  An element $x^n$ of $\Delta$ identifies with an element of $S^1 \times \Omega$ by taking $y^n \in \Omega$ to be the sequence of constant values of $x^n$, with $x^n(0) \equiv y^n_0, x^n(h) \equiv y^n_1,$ etc., and taking $\tau \in [0,h)$ to be the unique offset so that $x^n(t-\tau) \in \bar{\Delta}$.

$S^1$ is compact.  Therefore, given the sequence $\{x^n\} \in S^1 \times \Omega$, there exists a subsequence $\{x^{n_k}\}, k \in \mathbb{N}$ for which the offsets $\{\tau^{n_k}\}$ converge to a value in $[0,h)$.  Therefore, there exists a convergent subsequence $\{\tau^{n_k}\}$ of $\{\tau^n\}$.

For this subsequence $\{x^{n_k}\}$, we want to show that there exists a subsequence $\{x^{n_{k_j}}\}, j \in \mathbb{N}$ such that the bi-infinite sequences $\{y^{n_{k_j}}\}$ converge.  We do this inductively, beginning with the subsequence $\{y^{n_{k_j}}_0\}$.  We know that  $\{y^{n_k}_0\}$ is an infinite sequence of finitely many values, since the state space $S$ is finite.  Therefore, by the pigeonhole principle, there is one value that is repeated infinitely many times.  Take $\{y^{n_{k_j}}_0\}$ to be this value, $s_0$, so that $x^{n_{k_j}}(0) = s_0$ for all $j \in \mathbb{N}$.

Now, we induct.  Given a subsequence of $\{x^{n_k}\}$ that converges at $t = 0, h, -h, 2h, -2h, ..., $ $mh, -mh$, we deduce that there must be a subsequence of this subsequence with one value $x^{n_k}((m+1)h) = s_{(m+1)h}$ repeated infinitely many times, and likewise for $x^{n_k}(-(m+1)h) = s_{-(m+1)h}$.  In this manner, we get an infinite subsequence $\{y^{n_{k_j}}\}$, hence $\{x^{n_{k_j}}\}$, converging to a function that is piecewise constant on $[\tau + nh, \tau + (n+1)h), n \in \mathbb{Z}, \tau \in [0,h]$, with values in $S$.


Finally, we have to show closure.  That is, we need to show that transitions $x(mh) \rightarrow x((m+1)h)$ in our limit function are allowable.  Otherwise, all we would have shown is compactness of $\bar{\Lambda}$, rather than compactness of $\Delta$.  Suppose that $x \notin \Delta$.  Then, there exists some $m \in \mathbb{Z}$ such that the transition $x(mh) \rightarrow x((m+1)h)$ is not allowed.  But, since $\{x^{n_{k_j}}\}$ converges to $x$, we can take $N$ large enough so that $j > N \Rightarrow x^{n_{k_j}}(mh) = x(mh), x^{n_{k_j}}((m+1)h) = x((m+1)h)$.  And, $x^{n_{k_j}} \in \Delta$, so the transition $x^{n_{k_j}}(mh) \rightarrow x^{n_{k_j}}((m+1)h)$ must be allowable.  This is a contradiction.  Therefore, $\{x^{n_{k_j}}\} \rightarrow x \in \Delta$.

\end{proof}

\subsection{Morse Sets and Topological Chaos}


For the following definitions and Proposition \ref{order}, taken from \cite{Kliemann}, let $X$ be a compact metric space with an associated flow $\Phi$.

\begin{defn}\label{invariant}
A set $K \subseteq X$ is called invariant if $\Phi(t,x) \in K$ for all $x \in K, t \in \mathbb{R}$.
\end{defn}

\begin{defn}\label{isolated}
A set $K \subseteq X$ is called isolated if there exists a neighborhood $N$ of $K$ (i.e. a set $N$ with $K \subset$ int $N$) such that $\Phi(t,x) \in N$ for all $t \in \mathbb{R}$ implies $x \in K$.
\end{defn}

\begin{defn}\label{morsedecomp}
A Morse Decomposition on $X$ is a finite collection $\{\mathcal{M}_i, i = 1,...,n\}$ of non-void, pairwise disjoint, invariant, isolated, compact sets such that
\begin{enumerate}
\item For all $x \in X, \omega(x), \alpha(x) \subseteq \displaystyle \bigcup_{i=1}^{n}\mathcal{M}_i$.
\item If there exist $\mathcal{M}_0, \mathcal{M}_1, ..., \mathcal{M}_l$ and $x_1,...x_l \in X \setminus \displaystyle \bigcup_{i=1}^{n}\mathcal{M}_i$ with $\alpha(x_j) \subseteq \mathcal{M}_{j-1}$ and $\omega(x_j) \subseteq \mathcal{M}_j$ for $j = 1,...,l$, then $\mathcal{M}_0 \neq \mathcal{M}_l$.  This condition is equivalent to to the statement that there are no cycles between the sets of the Morse decomposition.
\end{enumerate}
The sets $\mathcal{M}_i$ above are called Morse sets.
\end{defn}

\begin{prop}\label{order}
The relation $\preceq$ given by
$$
\mathcal{M}_i \preceq \mathcal{M}_k \mbox{ if there are } \mathcal{M}_i,\mathcal{M}_{j_1},...,\mathcal{M}_{j_l} =
 \mathcal{M}_k \mbox{ and } x_1,...,x_l \in X $$
 $$\mbox{ with } \alpha(x_m) \subseteq \mathcal{M}_{j_{m-1}} \mbox{ and } \omega(x_m) \subseteq \mathcal{M}_{j_m} \mbox{ for } m = 1,...,l.
$$
is an order (satisfying reflexivity, transitivity, and antisymmetry) on the Morse sets $\mathcal{M}_j$ of a Morse decomposition.
\end{prop}

The proof of this proposition can be found in \cite{Kliemann}.


\begin{defn}\label{lifts}
The lift $\bar{\Delta}_C \subseteq \bar{\Delta}$ of a communicating class $C$ is defined by
$$
\bar{\Delta}_C \equiv \{ f \in \bar{\Delta} | f(t) \in C \, \forall \, t \in \mathbb{R} \}
$$
$\Delta_C$ is defined as 
$$
\Delta_C \equiv \bar{\Delta}_C \cap \Delta
$$
\end{defn}

\begin{thm}\label{morse1}
The lifts of the communicating classes $\bar{\Delta}_C$ are Morse sets for the dynamical system $\psi$.
\end{thm}

\begin{proof}
We check the seven conditions in turn.
\begin{enumerate}
\item \textbf{Non-void}
Since the empty set is not a communicating class, the lift of any communicating class must be non-empty.
\item \textbf{Pairwise disjoint}
Suppose that there exists $f \in \bar{\Delta}_C, \bar{\Delta}_{C'}$ with $C \neq C'$.  Then, $f(0) \in C, C'$.  But, by the maximality of communicating classes, $f(0) \in C, C'$ implies $C = C'$.  So, $\bar{\Delta}_{C} = \bar{\Delta}_{C'}$.
\item \textbf{Invariant}
By construction of $\bar{\Delta}$, $\psi(t,f) \in \bar{\Delta}$ for all $t \in \mathbb{R}, f \in \bar{\Delta}$.  And, if $f(s) \in C$ for all $s \in \mathbb{R}$, then $\psi(t,f)(s) = f(t+s) \in C$.  So, $\psi(t,f) \in \bar{\Delta}_C$ for all $t \in \mathbb{R}, f \in \bar{\Delta}_C$.
\item \textbf{Isolated}
Pick $\epsilon = 1/4$.  Suppose that there exists $g \notin \bar{\Delta}_C$ such that for some $f \in \bar{\Delta}_C$, $d(g,f) < \epsilon$.  Since $g \notin \bar{\Delta}_C$, there exists $t_0$ such that $g(t_0) \notin C$.  Let $g' = \psi(-t_0,g)$, so that $g'(0) \notin C$.  But then, $g'$ differs from any function in $\bar{\Delta}_C$ on at least some interval of length $h$ containing $0$.  The distance, therefore, between $g'$ and any function in $\bar{\Delta}_C$ must be greater than $1/4$.  So, given any $g \notin \bar{\Delta}_C$ but within $1/4$ of $\bar{\Delta}_C$, there exists $t_0$ such that $d(\psi(t_0, g),f') > \frac{1}{4}$ for any $f' \in \bar{\Delta}_C$.  Hence, $\bar{\Delta}_C$ is isolated.
\item \textbf{Compact}
By an argument similar to that for compactness of $\Delta$ and by compactness of $\Omega_C$, $\bar{\Delta}_C$ is compact.

\item \textbf{No cycles}
Again, this is similar to the corresponding proof in \cite{discretesystems}.  Suppose that there exist $f, g \in \bar{\Delta}$ such that $\alpha(f) \subseteq \bar{\Delta}_C$, $\alpha(g) \subseteq \bar{\Delta}_{C'}$ and $\omega(g) \subseteq \bar{\Delta}_C$, $\omega(f) \subseteq \bar{\Delta}_{C'}$.  Then, since all the transitions in $f,g$ must be allowable, there must exist an admissible path from $C$ to $C'$ as well as one from $C'$ to $C$.  But, this contradicts maximality of communicating classes.  So, no such cycle exists.
\end{enumerate}

\end{proof}


\begin{defn}\label{topologicallytransitive}
A flow on a metric space $X$ is called topologically transitive if there exists $x \in X$ such that $\omega(x) = X$.
\end{defn}

\begin{lem}\label{omegalimit}
Given any communicating class $C$, there exists $f^* \in \bar{\Delta}_C$ such that $\omega(f^*) = \bar{\Delta}_C$ (i.e. $\psi$ is topologically transitive on lifts of communicating classes).
\end{lem}

\begin{proof}
It has been shown in \cite{discretesystems} that for the discrete system, there exists $x^* \in \Omega_C$ such that $\omega(x^*) = \Omega_C$.  By the correspondence between sequences in $\Omega$ and functions in $\Delta$, there exists $f^* \in \Delta_C$ given by $$ f^*(nh) = x^*_n, n \in \mathbb{Z}$$ such that $\Delta_C \subseteq \omega(f^*)$.  And, since $\bar{\Delta}_C$ is given by the shifts $\psi(t,\Delta_C)$, it is clear that $\bar{\Delta}_C \subseteq \omega(f^*)$.  $\bar{\Delta}_C$ is invariant by Theorem \ref{morse1}, so $\omega(f^*) \subseteq \bar{\Delta}_C$, and $\omega(f^*) = \bar{\Delta}_C$.
\end{proof}

Since the $\omega$-limit sets of a point on a compact space are connected, we get the following corollary.

\begin{cor}\label{connected}
$\bar{\Delta}_C$ is connected.
\end{cor}

\begin{prop}\label{dense}
The set of all functions $f^*$ satisfying $\omega(f^*) = \bar{\Delta}_C$ is dense in $\bar{\Delta}_C$
\end{prop}

\begin{proof}
Given $f \in \bar{\Delta}_C$, there exists $f^*$ such that $f \in \omega(f^*)$, by Lemma \ref{omegalimit}.  Therefore, given $\epsilon > 0$, there exists $t \in \mathbb{R}$ such that $d(\psi(t,f^*),f) < \epsilon$.   $\omega(f^*) = \bar{\Delta}_C$ implies $\omega(\psi(t,f^*)) = \bar{\Delta}_C$. So, for any $f \in \bar{\Delta}_C$ and any $\epsilon > 0$, there exists a function $\psi(t,f^*) \in \bar{\Delta}_C$ with $d(\psi(t,f^*),f) < \epsilon$ and $\omega(\psi(t,f^*)) = \bar{\Delta}_C$.
\end{proof}

\begin{defn}\label{finer}
A Morse Decomposition $\{\mathcal{M}_1,...,\mathcal{M}_n\}$ is called finer than a Morse Decomposition $\{\mathcal{M}'_1,...,\mathcal{M}'_l\}$ if for all $j \in \{1,...,l\}$ there exists $i \in \{1,...,n\}$ such that $\mathcal{M}_i \subseteq \mathcal{M}'_j$, where containment is proper for at least one $j$.
\end{defn}

\begin{thm}\label{finest}
The lifts of the communicating classes $\bar{\Delta}_C$ form a finest Morse decomposition on $\bar{\Delta}$.
\end{thm}

\begin{proof}
Suppose there exists a finer Morse decomposition.  Then, for some $C$, there exists a Morse set $K \subsetneq \bar{\Delta}_C$, a proper containment.  Suppose first that there is only one such $K$.  By the definition of a Morse set, $K$ must contain the $\omega$-limit sets of $\bar{\Delta}_C$.

But, by Lemma \ref{omegalimit}, there exists $f^* \in \bar{\Delta}_C$ such that $\omega(f^*) = \bar{\Delta}_C$.  Therefore, $\bar{\Delta}_C \subseteq K$, so $K$ is not a proper subset of $\bar{\Delta}_C$.

Now, suppose that for this finer Morse decomposition, there exist several Morse sets $K_1,...,K_n$ such that each $K_i \subsetneq \bar{\Delta}_C$.  By the definition of a Morse set, $K_1,...K_n$ are pairwise disjoint and compact.  But, since $\omega(f^*) = \bar{\Delta}_C$, we must have $\displaystyle \bigcup_{i=1}^nK_i = \bar{\Delta}_C$.  By Corollary \ref{connected}, $\bar{\Delta}_C$ is connected.  And, the union of finitely many, pairwise disjoint compact sets cannot be connected.  Thus, no finer Morse decomposition exists.

\end{proof}

\begin{defn}\label{sensitivedependence}
A flow $\Phi$ on a metric space $X$ has sensitive dependence on initial conditions if there exists $\delta > 0$ such that for every $x \in X$ and every neighborhood $B$ of $x$, there exists $y \in B$ and $t > 0$ such that $d(\Phi_t(x),\Phi_t(y)) > \delta$.
\end{defn}

\begin{defn}\label{chaotic}
A flow on a metric space $X$ is chaotic if it has sensitive dependence on initial conditions and is topologically transitive.
\end{defn}

\begin{lem}\label{sensdep}
Consider a graph $G$ consisting of a single communicating class $C$ for which the out-degree of at least one vertex is at least two.  Then, $\psi$ on $\bar{\Delta}$ has sensitive dependence on initial conditions.

\end{lem}

\begin{proof}
Take $\delta = \frac{1}{2}$.  Given $x \in \bar{\Delta}$, we construct a function $y \in \bar{\Delta}$ such that $x$ and $y$ are discontinuous at the same times mod $h$.  Given $\varepsilon > 0$, take $N$ large enough so that
$$\displaystyle \sum_{i=-\infty}^{-N}{\frac{1}{4^{|i|}}} + \displaystyle \sum_{i=N}^{\infty}{\frac{1}{4^{|i|}}} < \varepsilon.$$
Thus, taking $x(t) = y(t)$ on $t \in [-Nh,Nh]$ ensures that $d(x,y) < \varepsilon$.  Now, we just need to show that there exists $m>N \in \mathbb{R}$ so that $y(t) \neq x(t)$ for all $t \in [mh,(m+1)h)$.  This would imply that $d(\psi(mh,x),\psi(mh,y)) \geq 1 > \delta$.
To show that such an $m$ exists, let $\gamma_1$ denote the vertex with out-degree greater than one.  If there does not exist a $t > Nh$ such that $x(t) = \gamma_1$, then given $x(Nh) = \gamma_N$, let $y(t), t > Nh$ follow a path from $\gamma_N$ to $\gamma_1$.  Such a path must exist since $G$ consists of a single communicating class, so there exists a path between any two vertices in $G$.  Thus, we would have $x(t_0+t) \neq \gamma_1, y(t_0+t) = \gamma_1$ for some $t_0 > Nh, t \in [0,h)$, so we can take $m = \frac{t_0}{h}$. If there does exist a $t > Nh$ such that $x(t) = \gamma_1$, then define $\gamma_2 = x(t+h)$, and take $t_1$ so that $x(t_1+t)=\gamma_1$ for $t \in [0,h)$.  Since the out-degree of $\gamma_1$ is greater than one, there exists an edge from $\gamma_1$ to some other vertex, $\gamma_3$ (note that it is possible that either $\gamma_1 = \gamma_2$ or $\gamma_1 = \gamma_3$, but not that $\gamma_2 = \gamma_3$).  Set $y(t_1+h) = \gamma_3$, and take $m = \frac{t_1+h}{h}$.

\end{proof}

Applying Lemmas \ref{omegalimit} and \ref{sensdep} yields the following result.

\begin{thm}\label{chaos}
Consider a graph consisting of a single communicating class for which the out-degree of at least one vertex is greater than one.  Then, $\psi$ is chaotic on $\bar{\Delta}$.
\end{thm}



\section{The Deterministic Hybrid System}

Now that the behavior on $\Delta$ has been determined given a natural number $n$ and an $N$-graph on $n$ vertices, we consider the action of a function $f\in\Delta$ on a set of $n$ dynamical systems.

Consider an $N$-graph $G$ with $n$ vertices.  Take a collection of $n$ dynamical systems $\{\phi_1,...,\phi_n\}$ on a compact space $M \subset \mathbb{R}^d$, where each vertex of $G$ corresponds to one dynamical system $\phi_i$.  Take $f \in \bar{\Delta}$.  Define $\varphi(t,x,f):\mathbb{R} \times M \times \bar{\Delta} \rightarrow M$ by
$$
\varphi n(t,x,f) = \varphi_t(x,f)= \phi_{i_m}(\tau_{m}-\tau_{m-1},\phi_{i_{m-1}}(\tau_{m-1}-\tau_{m-2},...\phi_{i_{1}}(\tau_1,x))) 
$$
where the $\tau_k$'s satisfy $0 = \tau_0 < \tau_1 <...< \tau_m = t$ and $f(\tau) = i_j$ for $\tau \in [\tau_{j-1},\tau_j)$.

Thus, $\varphi(t,x,f)$ is given by the flow along the dynamical system $\phi_i$ during the period of time for which $f = i$.  With this, we can explicitly define our deterministic hybrid system.  Consider 

\begin{center}
$\Phi_t(x_0,f_0) \equiv \left (
\begin{array}{cc}
\varphi_t(x_0,f_0) \\
\psi_t(f_0)
\end{array} \right ) $
\end{center}

with initial conditions $f_0 \in \bar \Delta$, and $x_0\in M$.  Let $\psi_t(f_0) = f_t$ and notice that

\begin{center}
$\Phi_0(x_0,f_0) = \left (
\begin{array}{cc}
x_0 \\
f_0
\end{array} \right ) $
\end{center}

Then,
$$\Phi_{t+s}(x_0,f_0) 
= 
\left (
\begin{array}{cc}
\varphi_{t+s} (x_0, f_0) \\
f_{t+s}
\end{array} \right ) 
 = 
\left (
\begin{array}{cc}
\varphi_t (\varphi_s(x_0, f_0), f_s) \\
f_t \circ f_s
\end{array} \right ) 
= 
\Phi_t \circ \Phi_s(x_0,f_0) .$$
Thus, $\Phi_t$ is in fact a flow, so the deterministic hybrid system is a dynamical system.

\subsection{Limit Sets on $M \times \bar{\Delta}$}

Part 6 of Theorem \ref{morse1} yields the following corollary.

\begin{cor}\label{morse2}
Given $y \in M \times \bar{\Delta}$, the $\bar{\Delta}$ component of $\alpha(y)$ is contained in $\bar{\Delta}_C$ for some some communicating class $C$, and the $\bar{\Delta}$ component of $\omega(y)$ is contained in $\bar{\Delta}_{C'}$ for some some communicating class $C'$.
\end{cor}

The implication of this corollary is that in order to find the $\alpha$/$\omega$-limit sets of $M \times \bar{\Delta}$, we only need study the $\alpha$/$\omega$-limit sets of trajectories whose states are contained entirely within one particular communicating class.

\begin{defn}\label{projection}
Given a set $X \subseteq M \times \bar{\Delta}$, we define the projection of $X$ onto $M$ by
$$
\pi_{M}(X) \equiv \{x \in M | \exists \, f \mbox{such that } (x,f) \in X \}
$$
Define the projection of $X$ onto $\bar{\Delta}$ by
$$
\pi_{\bar{\Delta}}(X) \equiv \{f \in \bar{\Delta} | \exists \, x \mbox{such that } (x,f) \in X \}
$$.
\end{defn}


Given a hybrid system whose graph $G$ contains some vertex $\gamma$, with corresponding dynamical system $\phi_\gamma$ on $M$, define $\omega_\gamma(x), x \in M$ to be the $\omega$-limit set of $x$ for the flow $\phi_\gamma$.  Define $\alpha_\gamma(x)$ to be the $\alpha$-limit set of $x$ for $\phi_\gamma$.

\begin{lem}\label{selfloop2}
If, given a vertex $\gamma\in V$ has an edge from and to itself, then there exists $ f \in \bar{\Delta}$ such that
$$\omega _\gamma (x) = \pi_{M}(\omega(x,f))$$
$$\alpha _\gamma (x) = \pi_{M}(\alpha(x,f))$$
for all $x \in M$.
\end{lem}

\begin{proof}
Since there is an edge from $\gamma$ to itself, the function in $f\in\bar{\Lambda}$ given by $f\equiv\gamma$ is an element of $\bar{\Delta}$.  This corresponds to running $\phi_\gamma$ for all of time.  Clearly, the projection of the limit sets of this system onto $M$ are the precisely those of $\phi_\gamma$.

\end{proof}

\begin{thm}\label{selfloopomega}
Consider a graph $G$ consisting of a single communicating class $C$ for which every vertex has an edge starting and ending at itself and a corresponding finite set of dynamical systems $\{\phi_i\}$, each of which has a Morse decomposition $\{ \mathcal{M}_1^i,\mathcal{M}_2^i,..., \mathcal{M}_{n_i}^i \}$.  Then, for all $x \in M$, there exists $f \in \bar{\Delta}$ such that
$$
\pi_{M}(\omega(x,f)) \cap (\displaystyle \bigcup_j \mathcal{M}_j^i) \neq \emptyset \, \forall \, i .
$$
\end{thm}

\begin{proof}
 Given any $x_{00}$, any $\epsilon > 0$, and any $\phi_1$, there exists a time $t_{10}$ such that $\phi_1(t_{10},x_{00})$ is within $\varepsilon$ of an $\omega$-limit set of $\phi_1$ by definition of $\omega$-limit sets.  Since the Morse sets of $\phi_1$ contain the $\omega$-limit sets of $\phi_1$, being within $\epsilon$ of an $\omega$-limit set implies being within $\varepsilon$ of a Morse set.  So, let $f(t) = $ $\gamma_1$ for $t \in [0, t_{10})$, and let $\phi_1(t_{10},x_{00}) = x_{10}$.
Now, consider a second dynamical system $\phi_2$, represented in $G$ by the vertex $\gamma_2$.  Since the graph consists of a single communicating class, it is possible to transition from vertex $\gamma_1$ to vertex $\gamma_2$ in a finite time.  Again, pick $t_{20}$ so that $\phi_2(t_{20},x_{10}) \equiv x_{20}$ is within $\varepsilon$ of an $\omega$-limit set, hence Morse set, of $\phi_2$.  Iterate this process for all $n$ vertices in the communicating class, yielding a point $\phi_n(t_{n0},x_{(n-1)0}) \equiv x_{n0}$.

Next, repeat the above procedure for $\frac{\epsilon}{2}$.  That is, find $t_{11}$ such that $\phi_1(t_{11},x_{n0}) \equiv x_{11}$ is within $\frac{\epsilon}{2}$ of a Morse set of $\phi_1$, $t_{21}$ such that $\phi_2(t_{21},x_{11})$ is within $\frac{\epsilon}{2}$ of a Morse set of $\phi_2$, etc.  Repeat for $\frac{\epsilon}{4}, \frac{\epsilon}{8}, ...$.  In the limit, for each vertex $\gamma_i$, we get a sequence of points $\{x_{ij}\}, j = 0, 1, 2, ...$ such that $x_{ij}$ is within $\frac{\epsilon}{2^j}$ of a Morse set of $\phi_i$.  Since there are only finitely many Morse sets for each $\phi_i$, there exists a subsequence $\{x_{ij_k}\}$ of $\{x_{ij}\}$ converging to one particular Morse set $\mathcal{M}_l^i$ of $\phi_i$.  Thus, $\pi_{M}(\omega(x_{00},f)) \cap  \mathcal{M}_l^i \neq \emptyset .$

\end{proof}

Clearly, the equivalent theorem applies to $\alpha$-limit sets, going backwards in time.  The example of section $6.1$ demonstrates that the existence of self loops is necessary for this theorem to hold.

\subsection{Morse Sets on $M \times \bar{\Delta}$}

The following definitions and theorem come from \cite{Kliemann}.  Suppose $X$ is a compact space with associated flow $\Phi$.

\begin{defn}\label{epsilontchain}
Given $x,y \in X$ and $\epsilon, T > 0$, an $(\epsilon,T)$-chain from $x$ to $y$ is given by a natural number $n \in \mathbb{N}$ together with points
$$
x_0 = x,x_1,...,x_n=y \in X \mbox{ and times } T_0,...T_{n-1} \geq T
$$
such that $d(\Phi(T_i,x_i),x_{i+1})<\epsilon$ for $i = 0,1,...,n-1$.
\end{defn}

\begin{defn}\label{chainrecurrent}
A point $x \in X$ is chain recurrent if for all $\epsilon, T >0$ there exists an $(\epsilon,T)$-chain from $x$ to $x$.  A subset $Y \subseteq X$ is chain transitive if for all $x,y \in Y$ and all $\epsilon, T > 0$, there exists an $(\epsilon,T)$-chain from $x$ to $y$.
\end{defn}

\begin{thm}\label{chainrecurrentmorse}
Suppose $X$ has a finest Morse decomposition $\{\mathcal{M}_1,...,\mathcal{M}_n\}$.  Then, $x$ is chain recurrent for all $x \in \mathcal{M}_i$, and each $\mathcal{M}_i$ is connected.
\end{thm}

With this machinery, we are in a position to define a new term and state a theorem regarding the hybrid system.

\begin{defn}\label{attracting region}
Given a compact set $A \subset X$, we call $A$ an attracting region for a flow $\Phi$ if there exists an open neighborhood $N(A)$ of $A$ such that $\omega(N(A)) \subseteq A$.  We call a compact set $R$ a repelling region for $\Phi$ if there exists an open neighborhood $N(R)$ such that $\alpha(N(R)) \subseteq R$.
\end{defn}

\begin{thm}\label{alwaysattracting}
Suppose that $A \subsetneq M$ is an attracting region for every individual dynamical system $\phi_i, i \in \{1,...,n\}$ of a hybrid system.  Then, there exists a non-trivial Morse decomposition on $M \times \Delta$.  Furthermore, if there exists a finest Morse decomposition on $M \times \Delta$, then it contains a Morse set $\mathcal{M}_A$ such that $\pi_{M}(\mathcal{M}_A) \subseteq A$.
\end{thm}

\begin{proof}
Let $N(A) \equiv \displaystyle \bigcap_{i=1}^l{N_i(A)} \subset M$ be the intersection of the neighborhoods $N_i(A)$ around $A$ admitted by each of the flows $\phi_i$, and note that $N(A)$ is also an open neighborhood of $A$.  Clearly, if $A \subsetneq M$ is an attracting region for each flow $\phi_i$, then $A \times\Delta$ is an attracting region for the hybrid system flow $\Phi$, as we can take $N(A) \times \Delta$ to be the open neighborhood around $A \times \Delta$.

Given $x \in N(A)$, let $B$ be some closed neighborhood containing $A$ but contained in $N(A)$ so that $x \notin B$.  Let $C$ be some closed neighborhood properly containing $B$ but contained in $N(A)$ so that $x \in C$.  Let $\partial B$ and $\partial C$ denotes the boundaries of $B$ and $C$, respectively.  Note that $\partial B$ and $\partial C$ contain their boundaries and are bounded since $M$ is bounded, so that $\partial B$ and $\partial C$ are both compact.

We want to show that there exists a time $T$ such that $\varphi_T(y,f) \in B$ for all $y \in \partial C, f \in \Delta$.  Suppose not.  Then, for each $n \in \mathbb{N}$, there exists $t_n > n, x_n \in \partial C, f_n \in \Delta$ such that $\varphi_{t_n}(x_n,f_n) \in \partial B$.  This gives rise to sequences $\{t_n\}$, $\{x_n\}$, $\{ f_n\}$, and $\{\varphi_{t_n}(x_n,f_n)\}$ with $\{t_n\} \rightarrow \infty$.  By compactness of $\partial C$, $\Delta$, and $\partial B$, the sequences $\{x_n\}$, $\{ f_n\}$, and $\{\varphi_{t_n}(x_n,f_n)\}$ have convergent subsequences, so that there exists $\{ n_j \}$ with $\{t_{n_j}\} \rightarrow \infty$, $\{x_{n_j}\} \rightarrow \hat{x}$, $\{x_{n_j}\} \rightarrow \hat{x} \in \partial C, \{f_{n_j}\} \rightarrow f \in\Delta,$ and $\{\varphi_{t_{n_j}}(x_{n_j},f_{n_j})\} \rightarrow \bar{x} \in \partial B$.  But then, $\bar{x} \in \pi_{M}(\omega(N(A) \times \Delta))$.  Since $\bar{x} \in \partial B \subset N(A) \setminus A$, this contradicts the definition of an attracting region.

So, such a $T$ exists.  Take $\epsilon < $ min($d(x,B),d(C,B)$).  Since any trajectory starting in $C \setminus B$ will flow into $B$ by time $T$, and since the distance from any point in $B$ to $x$ is greater than $\epsilon$, no $(\epsilon,T)$-chain from $x$ to itself exists.  Therefore, $(x,f)$ is not chain recurrent for any $f \in \bar{\Delta}$.  But, if the only Morse decomposition were trivial, then all of $M \times \bar{\Delta}$ would be the only Morse set, hence chain recurrent by Theorem \ref{chainrecurrentmorse}.  This is a contradiction.  Thus, there exists a non-trivial Morse decomposition.

If it is known that a finest Morse decomposition exists on $M \times \bar{\Delta}$, then the fact that any solution beginning in $N(A)$ flows into $A$ and remains there implies that the intersection of $A$ with the projection of some limit set onto $M$ is nonempty.  Since $(N(A) \setminus A) \times \bar{\Delta}$ is not chain recurrent, as shown, $(N(A) \setminus A) \times \bar{\Delta}$ cannot be contained in a Morse set.  So, there exists a Morse set $\mathcal{M}_A$ such that $\pi_{M}(\mathcal{M}_A) \cap A \neq \emptyset$, $\pi_{M}(\mathcal{M}_A) \cap (N(A) \setminus A) = \emptyset$.  Since the Morse sets of a finest Morse decomposition are connected, this means that $\pi_{M}(\mathcal{M}_A) \subseteq A$.

\end{proof}

By the same argument, going backwards in time, we get the same result for repelling regions.

\begin{cor}\label{alwaysrepelling}
Suppose that $R \subsetneq M$ is a repelling region for every individual dynamical system $\phi_i, i \in \{1,...,l\}$ of a hybrid system.  Then, there exists a non-trivial Morse decomposition on $M \times \bar{\Delta}$.  Furthermore, if there exists a finest Morse decomposition on $M \times \bar{\Delta}$, then it contains a Morse set $\mathcal{M}_R$ such that $\pi_{M}(\mathcal{M}_R) \subseteq R$.
\end{cor}

Theorem \ref{alwaysattracting} and Corollary \ref{alwaysrepelling} have important ramifications for hybrid systems that consist of several, slightly-perturbed individual dynamical systems.  If each of these systems has an attracting or repelling fixed point in some neighborhood, then it is likely that some attracting or repelling region will exist that contains all of them.  We have shown that the existence of such a region implies a non-trivial Morse decomposition of the space.

Suppose now that the graph $G$ consists of a single communcating class, $C$, so that $\bar{\Delta} = \bar{\Delta}_C$.

\begin{defn}\label{attractorrepeller}
A compact invariant set $A$ is an attractor if it admits an open neighborhood $N$ such that $\omega(N) = A$.  A repeller is a compact invariant set $R$ that has an open neighborhood $N^*$ such that $\alpha(N^*) = R$.

\end{defn}

Note that any attractor or repeller must be contained in a Morse set, since Morse sets contain $\alpha$- and $\omega$-limit sets.  Furthermore, it can be shown that for any Morse decomposition, at least one Morse set is an attractor and at least one Morse set is a repeller \cite{Kliemann}.
 
\begin{thm}\label{allofdeltabar}
Suppose that an attractor $A \subseteq M \times \bar{\Delta}$ is a Morse set for some Morse decomposition.  Then, $\pi_{\bar{\Delta}}(A) = \bar{\Delta}$.
\end{thm}

\begin{proof}
Take $(\bar x, \bar f) \in A$.  Since $A$ is an attractor, there exists $\epsilon > 0$ and a ball $B$ of radius $\epsilon$ around $(\bar x, \bar f)$ such that the $\omega$ limit set of any point in the ball is contained in $A$.  By Lemma \ref{omegalimit}, there exists $f^* \in \bar{\Delta}$ so that $\omega(f^*) = \bar{\Delta}$.  Therefore, there exists $t$ such that $d(\psi(t,f^*),\bar f) < \epsilon$.  Thus, $(\bar x, \psi(t,f^*)) \in B$, so $\omega(\bar x, \psi(t,f^*)) \subseteq A$.  Since $\omega(f^*) = \bar{\Delta}$ implies $\omega(\psi(t,f^*)) = \bar{\Delta}$, we see that $\pi_{\bar{\Delta}}(A) = \bar{\Delta}$.

\end{proof}

The same argument applies for a repeller, yielding the following corollary.

\begin{cor}\label{repellertime}
Suppose that a repeller $R \subseteq M \times \bar{\Delta}$ is a Morse set for some Morse decomposition.  Then, $\pi_{\bar{\Delta}}(R) = \bar{\Delta}$.
\end{cor}

\section{Behavior Under Small Perturbations}
We now analyze the behavior of a dynamical system under ``small" perturbations, i.e. we compare the global behavior of%
\begin{equation}
\phi:\mathbb{R}\times M\rightarrow M \label{systunpert}%
\end{equation}
with that of
\begin{align}
\Phi &  :\mathbb{R}\times M\times\Delta\rightarrow M\times
\Delta\label{systpert}\\
\Phi(t,x,f)  &  \mapsto(\varphi(t,x,f),\psi(t,f))\nonumber
\end{align}
by looking at their respective Morse decompositions. Since the Morse
decompositions of (\ref{systunpert}) live in $M$, and those of (\ref{systpert}
) live in $M\times\Delta$, we need to make the decompositions
comparable. This is accomplished by projecting chain recurrent sets from
$M\times\Delta$ down to $M$ in the following way.

\begin{defn}
A set $E\subset M$ is called a chain set of $\Phi$, if

\begin{enumerate}
\item for all $x\in E$ there exists $f\in\Delta$ with
$\varphi(t,x,f)\in E$ for all $t\in\mathbb{R}$,

\item for all $x,y\in E$ and all $\epsilon,T>0$ there exists an
$(\epsilon,T)-$chain from $x$ to $y$,

\item $E$ is maximal with these two properties.
\end{enumerate}

Here an $(\epsilon,T)-$ chain from $x$ to $y$ is defined by $n\in N$,
$x_{0},...,x_{n}\in M$, $f_{0},...,f_{n-1}\in\bar{\Delta}$, and $t_{0}%
,...,t_{n-1}\geq T$ such that $x_{0}=x$, $x_{n}=y$, and $d(\varphi(t_{i}%
,x_{i},f_{i}),x_{i+1})\leq\epsilon$ for all $i=0,...,n-1$.
\end{defn}

\begin{defn}
The lift $\mathcal{L}(A)$ of a set $A \subseteq M$ onto $M \times \Delta$ is defined by
$$
\mathcal{L}(A) = \{(x,f)\in M \times \Delta | \pi_M(\Phi(t,x,f))\in A \, \forall \, t \in \mathbb{R})\} .
$$
\end{defn}

It turns out that chain sets in $M$ correspond exactly to the chain recurrent
components of $(\Phi,M\times\Delta)$.

\begin{thm}
Let $E\subset M$ be a chain set of $\Phi$.

\begin{enumerate}
\item The lift of $E$ to $\mathcal{E}\subset M\times\Delta$ is a
maximal, invariant chain transitive set of $(\Phi,M\times\Delta)$.

\item If $\mathcal{E}\subset M\times\Delta$ is a maximal, invariant
chain transitive set of $(\Phi,M\times\Delta)$ then its projection
$\pi_{M}(\mathcal{E})$ is a chain set.
\end{enumerate}
\end{thm}

\begin{proof}
(ii) Let $\mathcal{E}$ be an invariant, chain transitive set in $M\times
\bar{\Delta}$. For $x\in\pi_{M}(\mathcal{E})$ there exists $f\in
\bar{\Delta}$ such that $\varphi(t,x,u)\in\mathcal{E}$ for all
$t\in\mathbb{R}$ by invariance. Now let $x,\,y\in\pi_{M}(\mathcal{E})$ and
choose $\epsilon>0,\,T>0$. Then by chain transitivity of $\mathcal{E}$, we
can choose $x_{j},u_{j},t_{j}$ such that the corresponding trajectories
satisfy the required condition.  Note also that $E$ is maximal if and only if $\mathcal{E}$ is maximal.

(i) Let $(x,f),\;(y,g)\in\mathcal{E}$ and pick $\epsilon>0.$ Choose
$T_{0}>0$ large enough such that
\[
\sum_{j\notin\lbrack-T_{0},T_{0}]}^{{}}4^{-\left\vert j\right\vert
}<\epsilon
\]
and pick $T>T_{0}$. Since $\varphi(2T,x,f)\in E$ and $\varphi(-T,y,g)\in E$,
the fact that $E$ is a chain set yields the existence of $k\in\mathbb{N}$ and
$x_{0},...,x_{k}\in M,\,f_{0},...,f_{k-1}\in\bar{\Delta}%
,\,t_{0},...,t_{k-1}\geq T$ with $x_{0}=\varphi(2T,x,f),\,x_{k}%
=\varphi(-T,y,g),$ and
\[
d(\varphi(t_{j},x_{j},f_{j}),x_{j+1})<\epsilon\text{\ for }j=0,...,k-1.
\]
We now construct an $(\epsilon,T)$-chain from $(x,f)$ to $(y,g)$ in the
following way. Define
\[%
\begin{array}
[c]{cccc}%
t_{-2}=T, & x_{-2}=x, & g_{-2}=f, & \\
t_{-1}=T, & x_{-1}=\varphi(T,x,f), & g_{-1}(t)= & \left\{
\begin{array}
[c]{cc}%
f(t_{-2}+t) & \text{for }t\leq t_{-1}\\
f_{0}(t-t_{-1}) & \text{for }t>t_{-1}%
\end{array}
\right.
\end{array}
\]
and let the times $t_{0},...,t_{k-1}$ and the points $x_{0},...,x_{k}$ be as
given earlier. Furthermore, set
\[%
\begin{array}
[c]{ccc}%
t_{k}=T, & x_{k+1}=y, & g_{k+1}=g,
\end{array}
\]
and define for $j=0,...,k-2$%
\begin{align*}
g_{j}(t)  &  =\left\{
\begin{array}
[c]{ccc}%
g_{j-1}(t_{j-1}+t) & \text{for} & t\leq0\\
f_{j}(t) & \text{for} & 0<t<t_{j}\\
f_{j+1}(t-t_{j}) & \text{for} & t>t_{j},
\end{array}
\right. \\
g_{k-1}(t)  &  =\left\{
\begin{array}
[c]{ccc}%
g_{k-2}(t_{k-2}+t) & \text{for} & t\leq0\\
f_{k-1}(t) & \text{for} & 0<t\leq t_{k-1}\\
g(t-t_{k-1}-T) & \text{for} & t>t_{k-1},
\end{array}
\right. \\
g_{k}(t)  &  =\left\{
\begin{array}
[c]{cc}%
g_{k-1}(t_{k-1}+t) & \text{for }t\leq0\\
g(t-T) & \text{for }t>0.
\end{array}
\right.
\end{align*}
We see that
\[
(x_{-2},g_{-2}),\,(x_{-1},g_{-1}),...,(x_{k+1},g_{k+1})\text{ and }%
t_{-2},\,t_{-1},...,t_{k}\geq T
\]
yield an $(\epsilon,T)$-chain from $(x,f)$ to $(y,g)$ provided that for
$j=-2,\,-1,...,k$%
\[
d(g_{j}(t_{j}+\cdot),g_{j+1})<\epsilon.
\]
By choice of $T$ we have by the definition of the metric on $\bar{\Delta
}$ for all $h_{1},\,h_{2}\in\bar{\Delta}$
\begin{align*}
d(h_{1},h_{2})  &  =\sum_{j\in Z}4^{-\left\vert j\right\vert }f(h_{1}%
,h_{2},j)\\
&  =\sum_{j\in Z}4^{-\left\vert j\right\vert }\frac{1}{h}\int\limits_{jh}%
^{(j+1)h}\delta(h_{1},h_{2},t)dt\\
&  \leq\sum_{j\in\lbrack-T_{0},T_{0}]}4^{-\left\vert j\right\vert }\frac{1}%
{h}\int\limits_{jh}^{(j+1)h}\delta(h_{1},h_{2},t)dt\text{ \ }+\text{
}\epsilon\text{.}%
\end{align*}
Note that for the $g_{j}$, $j=-2,\,-1,...,k$ as defined above the integrands
vanish, giving the desired result.
\end{proof}

With this result, we can look at families of dynamical systems over graphs in
the following way:

Let $\alpha\in I$, where $I$ is some index set, and consider the family%
\begin{equation}
\Phi_{\alpha}:\mathbb{R}\times M\times\bar{\Delta}\rightarrow
M\times\bar{\Delta} \label{systfamdelta}%
\end{equation}
such that $\Phi_{\cdot}$ depends continuously on $\alpha$. We embed the
systems of (\ref{systfamdelta}) into a larger base space by defining:

Let $U\subset\mathbb{R}^{m}$ be compact, convex with $0\in$ int$(U)$, and set
$\mathcal{U}=\{u:\mathbb{R}\rightarrow U$, locally integrable$\}$. The space
$\mathcal{U}$ with the weak$^{\ast}-$topology is a compact space. If $G$ is a finite directed $N$-graph, then $\bar{\Delta}(G)\subset\mathcal{U}$. For the
family of systems%
\[
\Phi_{\alpha}:\mathbb{R}\times M\times\mathcal{U}\rightarrow M\times
\mathcal{U}%
\]
we have the following result.

\begin{lem}
\label{lemDyCon}Let $\alpha_{k}\rightarrow\alpha_{0}$ in $I$, and let
$E^{\alpha_{k}}\subset M$ be chain sets of $\Phi_{\alpha_{k}}$. If
\[
\underset{\alpha_{k}\rightarrow\alpha_{0}}{\lim\sup}E^{\alpha_{k}}\equiv\{x\in
M\text{, there exist }x^{\alpha_{k}}\in E^{\alpha_{k}}\text{ with }%
x^{\alpha_{k}}\rightarrow x\}\neq\emptyset
\]
then it is contained in a chain set $E^{\alpha_{0}}$ of $\Phi_{\alpha_{0}}$.
\end{lem}

For a proof of this result see Theorem 3.4.6 in \cite{DyCon}.

In many applications the parameter $\alpha\in I$ refers to the size of the
perturbation in the following way: Let $U\subset\mathbb{R}^{m}$ as above and
set $U^{\rho}\equiv\rho\cdot U$ for $\rho\geq0$. Assume that the compact state
space $M$ is a $C^{\infty}-$manifold and consider the family of differential
equations
\begin{equation}
\dot{x}=X(x,\rho f)\text{, }\rho\geq0\text{.} \label{systfamrho}%
\end{equation}
where $X$ is a $C^{\infty}-$vector field on $M$ and the parameter $\rho$ is the perturbation size.
Each differential equation (\ref{systfamrho}) gives rise to a parametrized
dynamical system on $M\times\mathcal{U}^{\rho}$ and on $M\times\bar
{\Delta}^{\rho}$. For $\rho=0$ this is the unperturbed system given
by $\dot{x}=X(x,0)$. In this set up, Lemma \ref{lemDyCon} reads:

\begin{cor}
\label{corDyCon}Let $\rho_{k}\rightarrow0$ and let $E^{\rho_{k}}\subset M$ be
chain sets of $\Phi_{\rho}:\mathbb{R}\times M\times\mathcal{U}^{\rho
}\rightarrow M\times\mathcal{U}^{\rho}$. If $\underset{\rho_{k}\rightarrow
0}{\lim\sup}E^{\rho_{k}}\neq\emptyset$ then it is a chain recurrent set of
$\Phi_{0}$.
\end{cor}

For a proof of this result see Corollary 3.4.8 in \cite{DyCon}. This corollary
then gives the following result for ``small" perturbations of dynamical systems
in the context of the systems on $M\times\mathcal{U}^{\rho}$:

\begin{cor}
\label{corDyCon2}Consider the family $\Phi_{\rho}:\mathbb{R}\times
M\times\mathcal{U}^{\rho}\rightarrow M\times\mathcal{U}^{\rho}$ of dynamical
systems given by the differential equations (\ref{systfamrho}). Assume that
$\Phi_{\rho}$ has a finest Morse decomposition in $M\times\mathcal{U}^{\rho}$
for $\rho\geq0$. Then there exists $\epsilon>0$ such that for all
$0\leq\rho\leq\epsilon$ the finest Morse decomposition of $\Phi_{\rho}$
corresponds to the one of $\Phi_{0}$, i.e. the Morse sets are in 1-1
correspondence and the order induced by the Morse decomposition of $\Phi_{0}$
is preserved.
\end{cor}

To obtain a similar result for dynamical systems over graphs, we consider the
systems $\Phi_{\rho}:\mathbb{R}\times M\times\bar{\Delta}^{\rho
}\rightarrow M\times\bar{\Delta}^{\rho}$ where the finite set $S^{\rho
}\subset U^{\rho}$. Note that chain sets over $\bar{\Delta}^{\rho}$ are
contained in those over $\mathcal{U}^{\rho}$. To make the systems $\Phi_{\rho
}$ comparable we require that $0\in S^{\rho}$ for all $\rho\geq0$. The
examples in the next section show limit sets and Morse sets of $\Phi_{\rho
}:\mathbb{R}\times M\times\bar{\Delta}^{\rho}\rightarrow M\times
\bar{\Delta}^{\rho}$ can behave quite differently from those of $\Phi
_{0}$. However, Corollary \ref{corDyCon2} implies the following result on
small perturbations via systems over graphs:

\begin{cor}
\label{corsmalldelta}Consider the family $\Phi_{\rho}:\mathbb{R}\times
M\times\bar{\Delta}^{\rho}\rightarrow M\times\bar{\Delta}^{\rho}$ of
dynamical systems given by the differential equations (\ref{systfamrho}).
Assume that $\Phi_{\rho}$ has a finest Morse decomposition in $M\times
\bar{\Delta}^{\rho}$ for $\rho\geq0$. Assume furthermore that $0\in
S^{\rho}$ for all $\rho\geq0$ and that the vertex of the underlying graph
corresponding to $0\in S^{\rho}$ has a self loop. Then there exists
$\epsilon>0$ such that for all $0\leq\rho\leq\epsilon$ the finest Morse
decomposition of $\Phi_{\rho}$ corresponds to the one of $\Phi_{0}$, i.e. the
Morse sets are in 1-1 correspondence and the order induced by the Morse
decomposition of $\Phi_{0}$ is preserved.
\end{cor}

The proof follows directly from the fact that, under the conditions of the
corollary, the Morse sets of the finest Morse decomposition of $\Phi_{0}$ are
contained in the chain sets over $\bar{\Delta}^{\rho}$ for each $\rho
\geq0$. Hence we obtain the results of Corollary \ref{corDyCon} for the
systems $\Phi_{\rho}:\mathbb{R}\times M\times\bar{\Delta}^{\rho
}\rightarrow M\times\bar{\Delta}^{\rho}$, which implies Corollary
\ref{corsmalldelta}.

\section{Examples}

The examples in this section show that even simple hybrid systems can admit complicated behavior.  Lemma \ref{selfloop2} and Theorem \ref{selfloopomega} suggest that it may be possible to define some sort of relationship between the $\omega$-limit sets of the individual flows $\{\phi_i\}$ and those of the hybrid system $\varphi$.  One might suspect that the limit sets of the dynamical system should be contained in the projection of the limit sets of the hybrid system onto $M$, or vice versa.  However, the example of section $6.1$ shows that these results are not true for a hybrid system on a general graph.

Conversely, the example in section $6.1$ indicates that it may be possible to find solutions of the hybrid system that avoid the limit sets of each of the individual dynamical systems.  The example in section $6.2$, however, shows that this need not be the case, even in the event that the graph is complete.

Finally, Theorem \ref{allofdeltabar} and Corollary \ref{repellertime} suggest that the projection of a Morse set of a hybrid system onto $\bar{\Delta}$ might be itself a Morse set of the dynamical system $\psi$ on $\bar{\Delta}$.  Section $6.3$ gives a counterexample to this conjecture.

\subsection{A Hybrid System that Introduces New Limit Sets}

Consider the hybrid system given by two flows, $A$, and $B$ on the compact space $M = [-1,1]$, with $\phi_A(t,x)$ given by the solution to 
$$
\dot{x} = (1-x)(1+x)(-\frac{1}{2}-x)
$$
and $\phi_B(t,x)$ given by the solution to 
$$
\dot{x} = (1-x)(1+x)(\frac{1}{2}-x)
$$
and graph $G$ given by Figure \ref{figure1}.

\begin{figure}[H]
\begin{center}
\includegraphics[width=15mm]{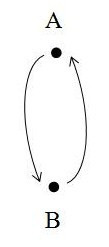}
\end{center}
\caption{Graph for the example hybrid system.}
\label{figure1}
\end{figure}

The fixed points for flow $A$ are $x = \pm 1, -\frac{1}{2}$, while the fixed points for flow $B$ are $x = \pm 1, \frac{1}{2}$.  Therefore, the $\omega$-limit sets for flows $A$ and $B$ are, respectively,
$$
\omega_A(M) = \{-1,-\frac{1}{2},1\}
$$
$$
\omega_B(M) = \{-1,\frac{1}{2},1\}
$$

As usual, let $h$ be the length of the intervals on which a function in $\bar{\Delta}$ is piecewise constant.

\begin{prop}\label{counterexample}
There exists $h$ so that given any $x \in M$ and any $f \in \bar{\Delta}$, the projection $\pi_{M}(\omega(x,f))$ contains neither $-\frac{1}{2}$ nor $\frac{1}{2}$.
\end{prop}

\begin{proof}
From the graph in Figure \ref{figure1}, we see that any function $f \in \bar{\Delta}$ will flicker back and forth between system $A$ and $B$.  During the intervals of time for which $f(t) = A$, solutions in $(-1,1)$ will be drawn toward the point $x = -\frac{1}{2}$.  Likewise, during the intervals for which $f(t) = B$, solutions in $(-1,1)$ will be drawn toward the point $x = \frac{1}{2}$.  It is clear, therefore, that given any $x \in (-1,1)$ and any $f \in \bar{\Delta}$, there exists a time at which the $M$ component of the solution will enter the region $(-\frac{1}{2},\frac{1}{2})$.

Suppose, then, that $x(0) \in (-\frac{1}{2},\frac{1}{2})$.  Whenever $f(t) = A$, $x$ will decrease, bounded below by $-\frac{1}{2}$.  Whenever $f(t) = B$, $x$ will increase, bounded above by $\frac{1}{2}$.  Thus, no solution that enters the interval $(-\frac{1}{2},\frac{1}{2})$ will ever leave it.

The only way that an $\omega$-limit set could contain either $-\frac{1}{2}$ or $\frac{1}{2}$, therefore, is if a solution approached one of the points asymptotically.  To eliminate this possibility, we need to show that there exists an $\epsilon > 0$ such that the distance between the solution $x$ and the points $-\frac{1}{2}$ and $\frac{1}{2}$ remains bounded below by $\epsilon$ in the limit as $t \rightarrow \infty$.  To do this, let $h$ be such that
$$\phi_A(h,\frac{1}{2}) < 0 .$$
By symmetry, this also implies that $\phi_B(h,-\frac{1}{2}) > 0$.  Now, pick a point $x \in (-\frac{1}{2},\frac{1}{2})$ and a $f \in \bar{\Delta}$, and allow the hybrid system run for a time $\tau$ so that
$$f(\tau + t) = \left\{ 
\begin{array}{cc}
A & \mbox{if} \, \left \lfloor \frac{t}{h} \right \rfloor \equiv 0 \mbox{ (mod 2)}\\
B & \mbox{otherwise} \\
\end{array} \right\}$$
This corresponds to shifting $f$ by a time $\tau$ such that $f(\tau) = A$ and $f$ is constant for the entire interval of length $h$ after $\tau$.  We know that at this time $\tau$, $x(\tau) \in (-\frac{1}{2},\frac{1}{2})$.  Thus, by our choice of $h$, $x(\tau + h) \in (-\frac{1}{2},0)$.  Let $\epsilon = \frac{1}{2} - \phi_B(h,0)$.  Since $x(\tau + h) \in (-\frac{1}{2},0)$, the distance between $x(\tau + 2h)$ and $\frac{1}{2}$ will be strictly greater than $\epsilon$.  But then, $x(\tau + 3h) \in (-\frac{1}{2},0)$, so $\frac{1}{2} - x(\tau + 4h) > \epsilon$.  More generally, $\frac{1}{2} - x(\tau + 2nh) > \epsilon$ for $n \in \mathbb{N}$.  And, since $x$ reaches a local maximum at each $\tau + 2nh$ before decreasing toward $-\frac{1}{2}$, the distance between $x$ and $\frac{1}{2}$ never approaches $0$ for $t > \tau$.  By symmetry, the distance between $x$ and $-\frac{1}{2}$ never approaches 0.  So, the projection of $\omega((x,f))$ onto $M$ contains neither $-\frac{1}{2}$ nor $\frac{1}{2}$ (c.f. Figure \ref{figure2}).

\end{proof}

\begin{figure}[H]
\begin{center}
\includegraphics[width=100mm]{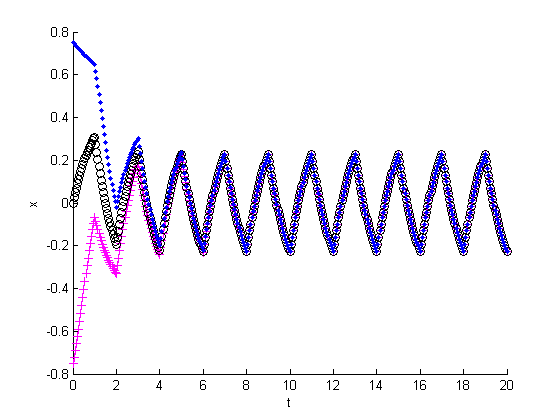}
\end{center}
\caption{Three representative trajectories for the example hybrid system.}
\label{figure2}
\end{figure}

This example has important implications.  Real-world systems are often modeled with dynamical systems by averaging over all possible states, neglecting the $\bar{\Delta}$ part of the equation.  This example shows that the limiting behavior of a hybrid system need not be anything like the limiting behavior of the individual dynamical systems comprising it.  The practice of ignoring $\bar{\Delta}$ may not always be valid.

\subsection{A Hybrid System with Limit Sets Containing Morse Sets}

Let $M = S^1$, the circle.  Suppose $G$ is a complete graph consisting of two vertices, with corresponding dynamical systems $\phi_1$ and $\phi_2$ shown in Figure \ref{figure3}.

\begin{figure}[H]
\begin{center}
\includegraphics[width=60mm]{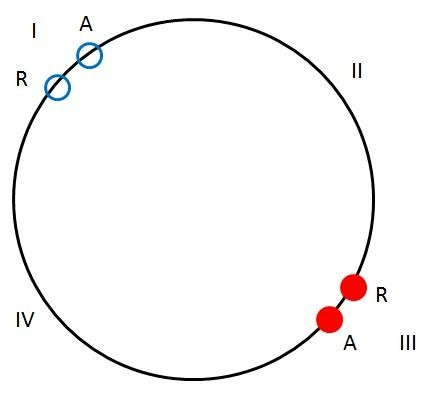}
\end{center}
\caption{The dynamical systems in the second example hybrid system.}
\label{figure3}
\end{figure}

In this picture, the letter $A$ denotes an attractor, while the letter $R$ denotes a repeller.  The red, filled circles correspond to the fixed points (and Morse sets $\mathcal{M}_j^1$) of $\phi_1$, while the blue, unfilled circles correspond to the fixed points (and Morse sets $\mathcal{M}_j^2$) of $\phi_2$.  The Morse sets divide the circle into four regions, which are marked I, II, III, and IV.

\begin{prop}\label{counterexample2}
There exists $h$ such that $\pi_{S^1}(\omega(x,f)) \cap (\displaystyle \bigcup_{i,j=1}^2 \mathcal{M}_i^j) \neq \emptyset$ for all $x \in S^1, f \in \Delta$.
\end{prop}



\begin{proof}
Consider $x\in S^1$ and $f\in\Delta$.  Notice that the flow in region II will always be counterclockwise, for both systems $\phi_1$ and $\phi_2$, and similarly the flow in region IV is always clockwise, for both systems $\phi_1$ and $\phi_2$.  Thus, trajectories in region II will precess toward region I and trajectories in region IV will precess toward region $III$, independent of the reigning function $f\in\Delta$.  Thus, we need only consider if the trajectory $x$ enters regions I or III. \\
\indent Assume there exists a time $\tau_2$ such that the projection of the trajectory of $(x,f)$ onto $S^1$ enters region III, and $f(t)=2$ for all $t>\tau_2$.  Thus, since in system $\phi_2$ solutions in region I converge to to $A_2$, $\pi_{S^1}(\omega(x,f))=A_2$.\\
\indent Similarly, if there exists a time $\tau_1$ such that the projection of the trajectory of $(x,f)$ onto $S^1$ enters region I and $f(t)=1$ for all $t>\tau_1$ then $\pi_{S^1}(\omega(x,f))=A_1$.\\
Now assume that the projection of the trajectory of $(x,f)$ enters region I at a time $t_2$, but that $f(t)\not = 2$ for all $t>t_1$.  Then there exists $s>t_2$ such that $f(s)$=1; that is, the system will eventually switch to the flow given by the system $\phi_1$, for a minimum of one time interval.  Then take $h_2$ large enough such that $\phi_1(A_2, h_2)$ is in region IV (which exists, since $A_2$ is not a fixed point of $\phi_1$ and there is no fixed point of $\phi_1$ between $A_2$ and region IV).  Then, even if $f=1$ for only one time interval, the trajectory of $(x,f)$ will be swept out of region I and into region IV.  Therefore, if any solution stays in region I indefinitely, it must converge to $A_2$.  \\
\indent Similarly, assume that the projection of the trajectory of $(x,f)$ enters region III at a time $t_1$ but that $f(t)\not =1$ for all $t>t_1$.  Thus there exists $s>t_1$ such that $f(s)=2$; that is, the system will eventually switch to the flow given by the system $\phi_2$, for a minimum of one time interval.  Take $h_1$ large enough such that $\phi_1(A_1,h_2)$ is in region II (which exists, since $A_1$ is not a fixed point of $\phi_2$ and there is no fixed point of $\phi_2$ between $A_1$ and region II).  Therefore, if any solution stays in region I indefinitely, it must converge to $A_1$. \\
\indent Take $h=\max (h_1,h_2)$.   In this case, no trajectory can remain in regions II or IV indefinitely, and if a trajectory remains in region I or III indefinitely, it must converge to either $A_2$ or $A_1$ respectively.  Thus, the only options for $\omega$-limit sets of $(x,f)$ are $\{A_1\}, \{A_2\}$ or all of $S^1$.  In every case, $$\pi_{S^1}(\omega(x,f)) \cap (\displaystyle \bigcup_{i,j=1}^2 \mathcal{M}_i^j) \neq \emptyset.$$

\end{proof} 
\subsection{Complete Characterization of Morse Sets of a Hybrid System}

Consider the 1-dimensional system on $[-1,1]$ given by
$$\dot{x} = (-x+1)(x+1)(x-(-1)^n1/2)^2$$
where $n$ takes values from the set $\{1,2\}$.  Let $G$, the directed graph dictating the changes in the values of $n$, be a complete graph on 2 vertices.  Then the dynamical system corresponding to $n=1$ has a repelling fixed point at $x=-1$, a saddle point at $x=-\frac{1}{2}$, and an attracting fixed point at $x=1$.  Meanwhile, the dynamical system corresponding to $n=2$ still has a repelling fixed point at $x=-1$ and an attracting fixed point at $x=1$, but its saddle is located at $x=\frac{1}{2}$.

\begin{prop}\label{thirdexample}

\begin{eqnarray*}
\mathcal{M}_1&=&\{-1\}\times\bar{\Delta}\\
\mathcal{M}_2&=&\{-1/2\}\times\{f \equiv 1\}\\
\mathcal{M}_3&=&\{1/2\}\times\{f \equiv 2\}\\
\mathcal{M}_4&=&\{1\}\times\bar{\Delta}
\end{eqnarray*}
is a Morse Decomposition for this hybrid system.
\end{prop}

\begin{proof}  We check the seven conditions in turn.
\begin{enumerate}
\item\textbf{Non-void}  Clearly, each Morse set described above is non-empty.

\item\textbf{Pairwise disjoint} Since the projections $\pi_M(\mathcal{M}_i)$ of the Morse sets onto $[-1,1]$ are pairwise disjoint, the Morse sets are pairwise disjoint.

\item\textbf{Invariant} Consider a point $(x,f)$ in $\mathcal{M}_1=\{-1\}\times\bar{\Delta}$.  Since the point $x = -1$ is invariant for both of the separate deterministic dynamical systems on $[-1,1]$, no matter what values $f$ takes, any trajectory will never leave $x = -1$.  Thus, since any translation of $f$ will always be in $\bar{\Delta}$, the trajectory of $(x,f)$ will always remain in $\mathcal{M}_1$ both forwards and backwards in time, and thus $\mathcal{M}_1$ is invariant.  By a similar argument, $\mathcal{M}_4$ is invariant.  \\
\indent Now consider the point $(-1/2, f)$, where $f \equiv 1$.   Since $f$ takes the value of $n=1$ for all time, this trajectory is equivalent to the trajectory that starts at $x=-1/2$ in the system given by $$\dot{x} = (-x+1)(x+1)(x+1/2)^2.$$
Since $x = -1/2$ is invariant in this system, the trajectory in the pair system will never leave $x = -1/2$.  Also, any shift of the function $f \equiv 1$ is still $f \equiv 1$, so $f$ is invariant in $\bar{\Delta}$.  Thus, $(-1/2, f)$ is invariant, so $\mathcal{M}_2$ is invariant.  By a similar argument for the system $$\dot{x} = (-x+1)(x+1)(x-1/2)^2,$$ $\mathcal{M}_3$ is invariant.

\item\textbf{Isolated}  A neighborhood of $\mathcal{M}_i$ given by $$N(\mathcal{M}_i)=\displaystyle\bigcup_{(x,f)\in\mathcal{M}_i}B((x,f),1/4)$$ is an open region containing $\mathcal{M}_i$ that does not intersect any other Morse set.  Consider a point $(x,f)\in N(\mathcal{M}_i)$ such that $\alpha(x,f)\subseteq N(\mathcal{M}_i)$ and $\omega(x,f)\subseteq N(\mathcal{M}_i)$.  Since by part 6 below every $\alpha,\omega$-limit set is contained in some Morse set, and since $\mathcal{M}_i$ is the only Morse set intersecting $N(\mathcal{M}_i)$, we have that $\alpha(x,f)\subseteq \mathcal{M}_i$ and $\omega(x,f)\subseteq \mathcal{M}_i$.  By part 7 below there are no cycles, so $(x,f)\in\mathcal{M}_i$.

\item\textbf{Compact} The sets $\{-1\},\{-1/2\},\{1/2\},\{1\}$ are compact in $M=[-1,1]$, and the sets $\bar{\Delta},\{f \equiv 1\},\{f \equiv 2\}$ are compact in $\bar{\Delta}$.  By Tychonoff's theorem, since each of the Morse sets $\mathcal{M}_i$ is a product of two of these sets, each is compact in the product space $M\times\bar{\Delta}$.

\item\textbf{Contains $\alpha$/$\omega$-limit sets}  By invariance of each Morse set, $(y,f)\in\mathcal{M}_i$ for some $i$ implies $\alpha(y,f)\subseteq\mathcal{M}_i$ and $\omega(y,f)\subseteq\mathcal{M}_i$.  Take a point $(y,f)$ with $y\in(-1,-1/2)$ so that there exists some $z\in\mathbb{R}$ such that $f(t)=1$ for all $t\geq z$.  Backwards in time, no matter what value $f$ takes, any trajectory moving in $(-1,-1/2)$ will converge towards $x=-1$, so $\alpha(y,z)\subseteq\mathcal{M}_1$.  Forwards in time, the function $f$ will converge to the function $g \equiv 1$.  So, suppose $f$ becomes constant before the projection of the trajectory into $[-1,1]$ has passed the point $x = -1/2$.  Since all trajectories starting in $(-1,1/2)$ will converge to the fixed point at $x = -1/2$ when $n=1$, the projection of the trajectory into $[-1,1]$ will converge to $x=-1/2$ in forwards time.  Thus, the trajectory in $M\times\bar{\Delta}$ will converge forwards to $\mathcal{M}_2$.  Instead, suppose $f$ becomes constant after the projection of the trajectory into $[-1,1]$ has passed the point $-1/2$.  Since all trajectories starting in $(-1/2,1)$ will converge to $x=-1/2$ when $n=1$, the projection of the trajectory into $[-1,1]$ will converge to $x=1$ in forwards time.  Thus, the trajectory in $M\times\bar{\Delta}$ will converge forwards to $\mathcal{M}_4$.  \\

Take a point $(y,f)$ where $y\in(1/2,1)$ and there exists some $z\in\mathbb{R}$ such that $f(x)=2$ for all $x\leq z$. Since, forwards in time, no matter what value $f$ takes, any trajectory moving in $(1/2,1)$ will converge towards $1$, $\omega(y,z)\subseteq\mathcal{M}_4$.  Backwards in time, clearly the function $f$ will converge to the function $g=2$.  Meanwhile, if $f$ becomes constant after the projection of the trajectory into $[-1,1]$ has passed the point $1/2$, since all trajectories starting in $(1/2,1)$ will converge backwards in time to the fixed point at $1/2$ when $n=2$, the projection of the trajectory into $[-1,1]$ will converge to the point $1/2$ in backwards time, and thus the trajectory in $[1,1]\times\bar{\Delta}$ will converge backwards to $\mathcal{M}_3$.  If $f$ becomes constant before the projection of the trajectory into $[-1,1]$ has passed the point $1/2$, since all trajectories starting in $(-1/2,1)$ will converge to the fixed point at $-1/2$ when $n=1$, the projection of the trajectory into $[-1,1]$ will converge to the point $-1/2$ in backwards time, and thus the trajectory in $[1,1]\times\bar{\Delta}$ will converge forwards to $\mathcal{M}_2$. \\

\indent Now consider a trajectory starting at $(y,f)$ where $y\in(-1,1)$ and there does not exist a $z$ such that $f(t)=1$ or $f(t)=2$ for all $t\geq z$ or $t\leq z$.  Then both forwards and backwards in time, the projection of this point's trajectory into $[-1,1]$ will pass the fixed points at $-1/2$ and at $1/2$, and thus will converge backwards to $\mathcal{M}_1$ and forwards to $\mathcal{M}_4$. 

Consider a trajectory starting at $(y,f)$, where $y\in(-1/2,1/2)$ where there exist $z_1$ and $z_2$, $z_1<z_2$ where $f(t)=1$ for all $t\leq z_1$ and $f(t)=2$ for all $t\geq z_2$. If $f$ becomes constant at $n = 1$ while the trajectory is in $(-1/2,1/2)$, then backwards in time the trajectory will converge towards $-1/2$ in $[-1,1]$, and the function $f$ will converge backwards to $g_1 \equiv 1$.  Thus $\alpha(y,f)=\mathcal{M}_2$.  However, if $f$ becomes constant before the trajectory has entered $(-1/2,1/2)$, then it will converge backwards in time to $\mathcal{M}_1$.  Meanwhile, forwards in time, if $f$ becomes constant at $n = 2$ while the trajectory is in $(-1/2,1/2)$, then forwards in time the trajectory will converge towards $1/2$ in $[-1,1]$, and the function $f$ will converge to $g_2 \equiv 2$.  Thus $\omega(y,f)=\mathcal{M}_3$.  However, if $f$ becomes constant at 2 after the trajectory has exited $(-1/2,1/2)$, then forwards in time the trajectory will converge to $\mathcal{M}_4$.  Thus, all $\alpha,\omega$-limit sets are contained in some $\mathcal{M}_i$.

\item\textbf{No Cycles} Note from the description above that given $\mathcal{M}_i$ and $\mathcal{M}_j$ where $i \neq j$ there do not exist points $(x,f)$ and $(y,g)$ in $M \times \bar{\Delta}$ such that $\alpha(x,f)\subseteq\mathcal{M}_i$ and $\omega(x,f)\subseteq\mathcal{M}_j$, and $\alpha(y,g)\subseteq\mathcal{M}_j$ and $\omega(y,g)\subseteq\mathcal{M}_j$.  Since the possibilities mentioned in part 6 are exhaustive, there are no cycles.

\end{enumerate}

\end{proof}

This example is important because $\pi_{\bar{\Delta}}(\mathcal{M}_{2,3}) \neq \bar{\Delta}$.  This implies that the projection of Morse sets on $M \times \bar{\Delta}$ onto $\bar{\Delta}$ do not need to be Morse sets on $\bar{\Delta}$.


\section{Conclusion}

We have seen that under the correct formulation, a hybrid system can be treated as a dynamical system on a compact space.  We have studied the Morse sets on $\bar{\Delta}$, showing that the lifts of invariant communicating classes form a finest Morse decomposition.  We have studied Morse decompositions of hybrid systems, concentrating on attracting and repelling Morse sets.  In the case that the random perturbations are small, we have seen that a hybrid system can be expected to behave similarly to the unperturbed dynamical system.  Finally, we have examined three examples of hybrid systems, which show that the limit sets of a hybrid dynamical system can be complicated objects, possibly with little relation to the limit sets of the individual dynamical systems comprising the hybrid system.  Future research could further illuminate the characteristics of the limit sets and Morse sets of these systems.


\section{Acknowledgments}

We wish to recognize Chad Vidden for his helpful discussions and Professor Wolfgang Kliemann for his instruction and guidance.  We would like to thank the Department of Mathematics at Iowa State University for their hospitality during the completion of this work.  In addition, we'd like to thank Iowa State University, Alliance, and the National Science Foundation for their support of this research.


\end{document}